\documentclass{amsart}
\usepackage{amsmath, amsthm, amscd, amsfonts, amssymb}

\makeatletter \oddsidemargin.9375in \evensidemargin \oddsidemargin
\newtheorem{thm}{Theorem}[section]

\newtheorem{cor}[thm]{Corollary}
\theoremstyle{definition}
\newtheorem{den}[thm]{Definition}
\newtheorem{example}[thm]{Example}
\theoremstyle{remark}
\newtheorem{rem}[thm]{Remark}
\numberwithin{equation}{section}

\begin{document}

\title[On approximate Connes-biprojectivity of dual Banach algebras] {On approximate Connes-biprojectivity of dual Banach algebras}
\author[A. Pourabbas]{A. Pourabbas}
\email{arpabbas@aut.ac.ir}

\author[A. Sahami]{A. Sahami}
\address{Department of Mathematics Faculty of Basic Science, Ilam University, P.O. Box 69315-516 Ilam, Iran.}
\email{a.Sahami@ilam.ac.ir}

\author[S. F. Shariati]{S. F. Shariati}

\address{Faculty of Mathematics and Computer Science,
    Amirkabir University of Technology, 424 Hafez Avenue, 15914
    Tehran, Iran.}

\email{f.Shariati@aut.ac.ir}



\begin{abstract}
In this paper, we introduce a notion of approximate Connes-biprojectivity for dual Banach  algebras. We study the relation between approximate Connes-biprojectivity, Johnson pseudo-Connes amenability and $\varphi$-Connes amenability. We propose a criterion to show that some certain dual triangular Banach algebras are not approximately Connes-biprojective. Next we show that for a locally compact group $G$, the Banach algebra  $M(G)$ is approximately Connes-biprojective if and only if $G$ is amenable. 
Finally for an infinite commutative compact group $G$ we show that the Banach algebra $L^2(G)$ with convolution product is approximately Connes-biprojective, but it is not Connes-biprojective.
\end{abstract}

\maketitle

\section{Introduction and Preliminaries}
One of the most important notion in the theory of homological Banach algebras
is biprojectivity which introduced by Helemskii \cite{Hel:89}. Indeed, A Banach algebra $\mathcal{A}$ is called
biprojective, if there exists a bounded $\mathcal{A}$-bimodule morphism $\rho:\mathcal{A}\rightarrow\mathcal{A}\hat{\otimes}\mathcal{A}$, such that $\rho$ is a right inverse
for $\pi_{\mathcal{A}}$, where $\pi_{\mathcal{A}}:\mathcal{A}\hat{\otimes}\mathcal{A}\rightarrow \mathcal{A}$ is the product morphism which is given by $\pi_{\mathcal{A}}(a\otimes b)=ab,$ for every $a,b\in \mathcal{A}$. 

Recently,   approximate homological
notion like approximate biprojectivity and approximate biflatness of Banach algebras have been studied by Zhang \cite{zhan:99}.
Indeed, a Banach algebra $\mathcal{A}$ is called approximately biprojective, if there exists a net $(\rho_{\alpha})$
of continuous $\mathcal{A}$-bimodule morphisms from $\mathcal{A}$ into $\mathcal{A}\hat{\otimes}\mathcal{A}$ such that $\pi_{\mathcal{A}}\circ \rho_{\alpha}(a)\rightarrow a$ for every $a\in\mathcal{A}$. For more  information about approximate biprojectivity of some semigroup algebras, see \cite{sah:2016}.

 There exists a class of Banach algebras which is called dual Banach algebras.  This category of Banach algebras defined by Runde \cite{Runde:2001}. Let $\mathcal{A}$ be a Banach algebra. A Banach $\mathcal{A}$-bimodule  $E$ is called dual if there is a closed submodule ${E}_{\ast}$ of ${E}^{\ast}$ such that $E=(E_{\ast})^{\ast}$. The Banach algebra $\mathcal{A}$ is called dual if it is dual as a Banach $\mathcal{A}$-bimodule. A dual Banach $\mathcal{A}$-bimodule $E$ is normal, if for each $x\in{E}$ the module maps $\mathcal{A}\longrightarrow{E}$; ${a}\mapsto{a}\cdot{x}$ and ${a}\mapsto{x}\cdot{a}$ are $wk^\ast$-$wk^\ast$ continuous. Let $\mathcal{A}$ be a Banach algebra and let $E$ be a Banach $\mathcal{A}$-bimodule. A bounded linear map $D:\mathcal{A}\rightarrow{E} $ is called a bounded derivation if  $D(ab)=a\cdot{D(b)}+D(a)\cdot{b}$ for every $a,b\in\mathcal{A}$. A bounded derivation $D:\mathcal{A}\rightarrow{E}$ is called inner if there exists an element $x$ in $E$ such that $D(a)=a\cdot{x}-x\cdot{a}$ ($a\in{\mathcal{A}}$). A dual Banach algebra $\mathcal{A}$ is called Connes amenable if for every normal dual Banach $\mathcal{A}$-bimodule $E$, every $wk^\ast$-continuous derivation $D:\mathcal{A}\longrightarrow{E}$ is inner. For a given dual Banach algebra $\mathcal{A}$ and a Banach $\mathcal{A}$-bimodule $E$, $\sigma{wc}(E)$ denote the set of all elements $x\in{E}$ such that the module maps $\mathcal{A}\rightarrow{E}$; ${a}\mapsto{a}\cdot{x}$ and ${a}\mapsto{x}\cdot{a}$
are $wk^\ast$-$wk$-continuous, one can see that, it is a closed submodule of $E$, see \cite{Runde:2001} and \cite{Runde:2004} for more details.
Note that, since  $\sigma{wc}(\mathcal{A}_{\ast})=\mathcal{A}_{\ast}$, the adjoint of $\pi_\mathcal{A}$ maps $\mathcal{A}_{\ast}$ into $\sigma{wc}(\mathcal{A}\hat{\otimes}\mathcal{A})^{\ast}$. Therefore, $\pi^{\ast\ast}_\mathcal{A}$ drops to an $\mathcal{A}$-bimodule morphism $\pi_{\sigma{wc}}:(\sigma{wc}(\mathcal{A}\hat{\otimes}\mathcal{A})^{\ast})^{\ast}\longrightarrow\mathcal{A}$. Every element $M\in{(\sigma{wc}(\mathcal{A}\hat{\otimes}\mathcal{A})^{\ast})^*}$ satisfying
\begin{center}
    $a\cdot{M}=M\cdot{a}\quad$ and $\quad{a}\pi_{\sigma{wc}}M=a\quad(a\in{\mathcal{A}})$,
\end{center}
is called a $\sigma{wc}$-virtual diagonal for $\mathcal{A}$. Runde showed that a dual Banach algebra $\mathcal{A}$ is Connes amenable if and only if there exists a $\sigma{wc}$-virtual diagonal for $\mathcal{A}$ \cite[Theorem 4.8]{Runde:2004}.

 A dual Banach algebra $\mathcal{A}$ is called Connes-biprojective if there exists a bounded $\mathcal{A}$-bimodule morphism $\rho:\mathcal{A}\longrightarrow(\sigma{wc}(\mathcal{A}\hat{\otimes}\mathcal{A})^{\ast})^{\ast}$ such that $\pi_{\sigma{wc}}\circ\rho=id_{\mathcal{A}}$ . Shirinkalam and the second author showed that a dual Banach algebra $\mathcal{A}$ is Connes amenable if and only if $\mathcal{A}$ is Connes-biprojective and it has an identity \cite{Shi:2016}.

 Motivated by the definitions of approximate biprojectivity \cite{zhan:99} and  Connes-biprojectivity,
 we introduce a  new class of dual Banach algebras. 
 \begin{den}
    A dual Banach algebra $\mathcal{A}$ is called approximately Connes-biprojective, if there exists a (not necessarily bounded) net $(\rho_\alpha)_\alpha$ of continuous $\mathcal{A}$-bimodule morphisms from $\mathcal{A}$ into $(\sigma wc(\mathcal{A}\hat{\otimes}\mathcal{A})^*)^*$ such that $\pi_{\sigma wc}\circ \rho_{\alpha}(a)\rightarrow a$ for every $a\in\mathcal{A}$.
 \end{den}
It is clear that every Connes-biprojective dual Banach algebra is approximately Connes-biprojective and the same result holds for every approximately biprojective dual Banach algebra.

 In this paper we  study the  notion of approximately Connes-biprojectivity of dual Banach algebras. We show that there exists a relation between this new notion and $\varphi$-Connes amenability. Using this criterion, we  investigate   approximate Connes-biprojectivity of  triangular Banach algebras. We study approximate Connes-biprojectivity  of some dual  Banach algebras associated with locally compact groups. More precisely, we show that for a locally compact group $G$, the measure algebra $M(G)$ is approximately Connes-biprojective if and only if $G$ is amenable. We extend the Example in \cite[\textsection2]{zhan:99} to the approximately Connes-biprojective case and we show that   for an infinite commutative compact group $G$ the Banach algebra $L^2(G)$ with convolution product is approximately Connes-biprojective, but it is not Connes-biprojective.

\section{Approximate  Connes-biprojectivity}
Recently a modificated notion of amenability for Banach algebra (Connes amenability for dual Banach algebra) like Johnson pseudo-contractibility (Johnson pseudo-Connes amenability) introduced \cite{Sahami:2017}(\cite{Sha:17}), respectively. A dual Banach algebra $\mathcal{A}$ is called Johnson pseudo-Connes amenable, if there exists a not necessarily bounded net $(m_{\alpha})$ in $(\mathcal{A}\hat{\otimes}\mathcal{A})^{**}$ such that $\langle{T},a\cdot{m_{\alpha}}\rangle=\langle{T},{m_{\alpha}}\cdot{a}\rangle$ and $i^{\ast}_{\mathcal{A}_{\ast}}\pi_{\mathcal{A}}^{**}(m_{\alpha}){a}\rightarrow{a}$ for every $a\in{\mathcal{A}}$ and $T\in\sigma{wc}(\mathcal{A}\hat{\otimes}\mathcal{A})^{\ast}$, where $i_{\mathcal{A}_{\ast}}:\mathcal{A}_{\ast}\hookrightarrow\mathcal{A}^{\ast}$ is the canonical embedding \cite{Sha:17}.
\begin{rem}\label{R2.1}
	Let $\mathcal{A}$ be a dual Banach algebra and let $X$ be a Banach $\mathcal{A}$-bimodule. Since $\sigma wc(X^*)$ is a closed $\mathcal{A}$-submodule of $X^{*}$, we have a quotient map $q:X^{\ast\ast}\longrightarrow(\sigma{wc}(X^{\ast}))^{\ast}$ defined by $q(u)=u\vert_{\sigma wc(X^*)}$ for every $u\in{X^{**}}$.
\end{rem}
Note that a dual Banach algebra $\mathcal{A}$ has a bounded approximate identity if and only if $\mathcal{A}$ has an identity.
\begin{thm}\label{t2.2}
    Let $\mathcal{A}$ be a dual Banach algebra. Then the following statements hold:
    \begin{enumerate}
        \item [(i)]
        if $\mathcal{A}$ is approximately Connes-biprojective and has an identity, then $\mathcal{A}$ is Johnson pseudo Connes amenable.
        \item[(ii)] if $\mathcal{A}$ is Johnson pseudo Connes amenable, then $\mathcal{A}$ is approximately Connes-biprojective.
    \end{enumerate}
\end{thm}
\begin{proof}
    (i) Let $\mathcal{A}$ be an approximately Connes-biprojective dual Banach algebra. Then there exists a net of bounded $\mathcal{A}$-bimodule morphisms $ (\rho_\alpha)$ such that $\rho_\alpha : \mathcal{A}\rightarrow (\sigma wc(\mathcal{A}\hat{\otimes}\mathcal{A})^*)^*$ and $\pi_{\sigma wc}\circ \rho_{\alpha} (a)\rightarrow a$ for every $a\in{\mathcal{A}}$. Let $m_\alpha=\rho_\alpha (e)$, where $e$ is an identity for $\mathcal{A}$. Consider the net $(M_{\alpha})$ in $(\mathcal{A}\hat{\otimes}\mathcal{A})^{**}$ such that $q(M_{\alpha})=m_\alpha$, where $q:(\mathcal{A}\hat{\otimes}\mathcal{A})^{**}\rightarrow(\sigma wc(\mathcal{A}\hat{\otimes}\mathcal{A})^*)^*$ is a quotient map as in the Remark \ref{R2.1}. For every $a\in{\mathcal{A}}$ and $T\in{\sigma wc(\mathcal{A}\hat{\otimes}\mathcal{A})^*}$ we have
    \begin{equation*}
    \begin{split}
    \langle T,a\cdot M_\alpha-M_\alpha\cdot a\rangle&= \langle T,(a\cdot M_\alpha-M_\alpha\cdot a)\vert_{\sigma{wc}({\mathcal{A}}\hat{\otimes}{\mathcal{A}})^{\ast}}\rangle=\langle T,q(a\cdot M_\alpha-M_\alpha\cdot a)\rangle\\&=\langle T,a\cdot m_\alpha-m_\alpha\cdot a\rangle=\langle T,a\cdot \rho_\alpha (e)-\rho_\alpha (e)\cdot a\rangle\\&=\langle T,\rho_\alpha(a)-\rho_\alpha(a)\rangle=0.
    \end{split}
    \end{equation*}
    Since $i^{\ast}_{\mathcal{A}_{\ast}}\pi^{**}_{\mathcal{A}}=\pi_{\sigma{wc}}\circ q$ \cite[Remark 2.1]{Sha:17},
    \begin{equation*}
i^{\ast}_{\mathcal{A}_{\ast}}\pi_{\mathcal{A}}^{**}(M_\alpha)a=\pi_{\sigma wc}\circ q(M_\alpha)a=\pi_{\sigma wc}(m_\alpha)a=(\pi_{\sigma wc}\circ\rho_{\alpha}(e))a\rightarrow a.
\end{equation*}
So $\mathcal{A}$ is Johnson pseudo Connes amenable.\\
(ii) Suppose that $\mathcal{A}$ is Johnson pseudo Connes amenable. Then there exists a net $(M_{\alpha})$ in $(\mathcal{A}\hat{\otimes}\mathcal{A})^{**}$ such that $\langle{T},a\cdot{M_{\alpha}}\rangle=\langle{T},{M_{\alpha}}\cdot{a}\rangle$ and $i^{\ast}_{\mathcal{A}_{\ast}}\pi_{\mathcal{A}}^{**}(M_{\alpha}){a}\rightarrow{a}$ for every $a\in{\mathcal{A}}$ and $T\in\sigma{wc}(\mathcal{A}\hat{\otimes}\mathcal{A})^{\ast}$. Let $m_\alpha=q(M_{\alpha})$, where $q:(\mathcal{A}\hat{\otimes}\mathcal{A})^{**}\rightarrow(\sigma wc(\mathcal{A}\hat{\otimes}\mathcal{A})^*)^*$ is a quotient map as in the Remark \ref{R2.1}. For every $\alpha$, we define the map
$\rho_\alpha :\mathcal{A}\rightarrow  (\sigma wc(\mathcal{A}\hat{\otimes} \mathcal{A})^*)^*$ by $\rho_\alpha (a)= m_\alpha\cdot a$, for every $a\in{\mathcal{A}}$. Thus for every $ \alpha $ we have
$$\Vert \rho_\alpha (a) \Vert = \Vert   m_\alpha\cdot a \Vert \leq \Vert  m_\alpha  \Vert \Vert a \Vert.$$
So $ \rho_\alpha $ is bounded for every $\alpha$. Since $\langle{T},a\cdot{M_{\alpha}}\rangle=\langle{T},{M_{\alpha}}\cdot{a}\rangle$  for every $T\in\sigma{wc}(\mathcal{A}\hat{\otimes}\mathcal{A})^{\ast}$ and $a\in{\mathcal{A}}$,
\begin{equation*}
a\cdot m_\alpha=q(a\cdot M_\alpha)=a\cdot M_\alpha\vert_{\sigma{wc}({\mathcal{A}}\hat{\otimes}{\mathcal{A}})^{\ast}}=M_\alpha\cdot a\vert_{\sigma{wc}({\mathcal{A}}\hat{\otimes}{\mathcal{A}})^{\ast}}=q(M_\alpha\cdot a)=m_\alpha\cdot a.
\end{equation*}
Thus for every $a,b\in{\mathcal{A}}$
\begin{equation*}
a\cdot\rho_\alpha(b)=a\cdot(m_\alpha\cdot b)=(a\cdot m_\alpha)\cdot b=(m_\alpha\cdot a)\cdot b=m_\alpha\cdot(ab) =\rho_\alpha(ab),
\end{equation*}
and also $ \rho_\alpha(a)\cdot b=\rho_\alpha(ab)$. Hence $\rho_\alpha$ is an $\mathcal{A}$-bimodule morphism.
Since $i^{\ast}_{\mathcal{A}_{\ast}}\pi^{**}_{\mathcal{A}}=\pi_{\sigma{wc}}\circ q$ \cite[Remark 2.1]{Sha:17}, we have
\begin{equation*}
\pi_{\sigma wc}\circ \rho_\alpha (a)=\pi_{\sigma wc}(m_\alpha\cdot a)=\pi_{\sigma wc}(m_\alpha)a=\pi_{\sigma wc}\circ q(M_\alpha)a\rightarrow a\quad(a\in{\mathcal{A}}).
\end{equation*}
So $\mathcal{A}$ is approximately Connes-biprojective.
\end{proof}
 We recall that a Banach algebra $\mathcal{A}$ is left $\varphi$-contractible, where $\varphi$ is a linear multiplication functional on $\mathcal{A}$, if there
 exists $m\in{\mathcal{A}}$ such that $am=\varphi(a)m$ and $\varphi(m)=1$, for every $a\in{\mathcal{A}}$ \cite{Hu:09}. The notion of $\varphi$-Connes amenability for a dual Banach algebra $\mathcal{A}$, where $\varphi$ is a ${wk}^{\ast}$-continuous character on $\mathcal{A}$, was introduced by Mahmoodi  and some characterizations were given \cite{Mahmoodi:2014}. We say that $\mathcal{A}$ is $\varphi$-Connes amenable if there exists a bounded linear functional $m$ on  $\sigma{wc}({\mathcal{A}}^{\ast})$ satisfying $m(\varphi)=1$ and $m(f\cdot{a})=\varphi(a)m(f)$ for every $a\in{\mathcal{A}}$ and $f\in{\sigma{wc}({\mathcal{A}}^{\ast})}$. Ramezanpour showed that the concept of $\varphi$-Connes amenability is equivalent with left $\varphi$-contractible for a dual Banach algebra, where $\varphi$ is a ${wk}^{\ast}$-continuous character \cite[Proposition 2.3]{ram:18}.

 Note that if $\mathcal{A}$ is an approximately Connes-biprojective dual Banach algebra with an identity, then by Theorem \ref{t2.2} (i) and \cite[Proposition 2.6]{Sha:17}, $\mathcal{A}$ is $\varphi$-Connes amenable, where $\varphi$ is a $wk^*$-continuous character. In the following Theorem we show that with the weaker condition than being a unital, $\mathcal{A}$ is left $\varphi$-contractible. Thus $\mathcal{A}$ is $\varphi$-Connes amenable \cite[Proposition 2.3]{ram:18}.
\begin{thm}\label{t2.5}
    Let $\mathcal{A}$ be an approximately Connes-biprojective dual Banach algebra and let $\varphi\in{\Delta_{wk^*}(\mathcal{A})}$ such that $\ker\varphi=\overline{\mathcal{A}\ker\varphi}$. Then $\mathcal{A}$ is left $\varphi$-contractible.
\end{thm}
\begin{proof}
    Suppose that $\mathcal{A}$ is approximately Connes-biprojective. Then there exists a net of bounded $\mathcal{A}$-bimodule morphisms $ (\rho_\alpha)$ such that $\rho_\alpha : \mathcal{A}\rightarrow (\sigma wc(\mathcal{A}\hat{\otimes}\mathcal{A})^*)^*$ and $\pi_{\sigma wc}\circ \rho_{\alpha} (a)\rightarrow a$ for every $a\in{\mathcal{A}}$. Let $L=\ker\varphi$. Then $L$ is a closed ideal in $\mathcal{A}$ and also $\frac{\mathcal{A}}{L}$ is a Banach $\mathcal{A}$-bimodule with natural operation. Consider the bounded $\mathcal{A}$-bimodule morphism $id_{\mathcal{A}}\otimes q:\mathcal{A}\hat{\otimes}\mathcal{A}\rightarrow\mathcal{A}\hat{\otimes}\frac{\mathcal{A}}{L}$, where $q:\mathcal{A}\rightarrow\frac{\mathcal{A}}{L}$ is a quotient map. Therefore $(id_{\mathcal{A}}\otimes q)^*$ maps $\sigma wc(\mathcal{A}\hat{\otimes}\frac{\mathcal{A}}{L})^*$ into $\sigma wc(\mathcal{A}\hat{\otimes}\mathcal{A})^*$. Consequently, we obtain a $wk^*$-continuous $\mathcal{A}$-bimodule morphism
    \begin{equation*}
    \Theta:=((id_{\mathcal{A}}\otimes q)^*\vert_{\sigma wc(\mathcal{A}\hat{\otimes}\frac{\mathcal{A}}{L})^*})^*:(\sigma wc(\mathcal{A}\hat{\otimes}\mathcal{A})^*)^*\rightarrow(\sigma wc(\mathcal{A}\hat{\otimes}\frac{\mathcal{A}}{L})^*)^*.
    \end{equation*}
    Define $\eta_{\alpha}=\Theta\circ\rho_{\alpha}$. So $\eta_{\alpha}:\mathcal{A}\rightarrow(\sigma wc(\mathcal{A}\hat{\otimes}\frac{\mathcal{A}}{L})^*)^*$ is an $\mathcal{A}$-bimodule morphism for every $\alpha$. Since $L=\overline{\mathcal{A}L}$, for every $l\in{L}$ there exist two nets $(a_{\gamma})$ in $\mathcal{A}$ and $(l_{\gamma})$ in $L$ such that $l=\lim\limits_{\gamma}a_{\gamma}l_{\gamma}$. So 
    \begin{equation}\label{e2.1}
    \eta_{\alpha}(l)=\Theta\circ\rho_{\alpha}(l)=\lim\limits_{\gamma}\Theta\circ\rho_{\alpha}(a_{\gamma}l_{\gamma})=\lim\limits_{\gamma}\Theta(\rho_{\alpha}(a_{\gamma})\cdot l_{\gamma})\qquad(l\in{L}).
    \end{equation}
    Composing the canonical inclusion
    map $\mathcal{A}\hat{\otimes}\mathcal{A}\hookrightarrow(\mathcal{A}\hat{\otimes}\mathcal{A})^{\ast\ast}$ with the quotient map $q:(\mathcal{A}\hat{\otimes}\mathcal{A})^{\ast\ast}\rightarrow(\sigma\omega{c}({\mathcal{A}}\hat{\otimes}{\mathcal{A}})^{\ast})^{\ast}$ as in the Remark \ref{R2.1}, we obtain a continuous $\mathcal{A}$-bimodule map $\tau:\mathcal{A}\hat{\otimes}\mathcal{A}\longrightarrow(\sigma\omega{c}({\mathcal{A}}\hat{\otimes}{\mathcal{A}})^{\ast})^{\ast}$ which has a $wk^{*}$-dense range. We write $\bar{u}$ instead of $\tau(u)=\hat{u}\vert_{\sigma\omega{c}({\mathcal{A}}\hat{\otimes}{\mathcal{A}})^{\ast}}$ for every $u\in{\mathcal{A}\hat{\otimes}\mathcal{A}}$. Consider $\rho_{\alpha}(a_\gamma)$ as an element in $(\sigma\omega{c}({\mathcal{A}}\hat{\otimes}{\mathcal{A}})^{\ast})^{\ast}$ for every $\gamma$. So there exists a net $(u_{\beta}^\gamma)$ in $\mathcal{A}\hat{\otimes}\mathcal{A}$ such that $\rho_{\alpha}(a_\gamma)=wk^*\hbox{-}\lim\limits_{\beta}\overline{u_{\beta}^\gamma}$. Since $\Theta$ is $wk^*$-continuous, (\ref{e2.1}) implies that
    \begin{equation}\label{eq}
\eta_{\alpha}(l)=\lim\limits_{\gamma}\Theta((wk^*\hbox{-}\lim\limits_{\beta}\overline{u_{\beta}^\gamma})\cdot l_{\gamma})=\lim\limits_{\gamma}\Theta(wk^*\hbox{-}\lim\limits_{\beta}(\overline{u_{\beta}^\gamma\cdot l_{\gamma}}))=\lim\limits_{\gamma}wk^*\hbox{-}\lim\limits_{\beta}\Theta(\overline{u_{\beta}^\gamma\cdot l_{\gamma}}).
    \end{equation}
One can see that
\begin{equation}\label{e222}
\Theta(\bar{u})=(\widehat{id_{\mathcal{A}}\otimes q(u)})\vert_{\sigma wc({\mathcal{A}}\hat{\otimes}\frac{\mathcal{A}}{L})^*}\qquad(u\in{{\mathcal{A}}\hat{\otimes}{\mathcal{A}}}).
\end{equation}
Since for every $l\in{L}$, $q(l)=0$, $id_{\mathcal{A}}\otimes q(u_{\beta}^\gamma\cdot l_{\gamma})=0$. By (\ref{eq}) and (\ref{e222}), $\eta_{\alpha}(l)=0$. Thus $\eta_{\alpha}$ can be dropped on $\frac{\mathcal{A}}{L}$, for every $\alpha$. So we can see that $\eta_{\alpha}:\frac{\mathcal{A}}{L}\rightarrow(\sigma wc(\mathcal{A}\hat{\otimes}\frac{\mathcal{A}}{L})^*)^*$ is an $\mathcal{A}$-bimodule morphism. Define a character $\tilde{\varphi}$ on $\frac{\mathcal{A}}{L}$ by $\tilde{\varphi}(a+L)=\varphi(a)$ for every $a\in{\mathcal{A}}$. Consider a left $\mathcal{A}$-module morphism $id_{\mathcal{A}}\otimes\tilde{\varphi}:\mathcal{A}\hat{\otimes}\frac{\mathcal{A}}{L}\rightarrow\mathcal{A}$; $a\otimes(b+L)\mapsto\varphi(b)a$. For every $f\in{\mathcal{A}^*}$ and $a\in{\mathcal{A}}$ we have
\begin{equation}\label{e2.2}
(id_{\mathcal{A}}\otimes\tilde{\varphi})^*(f)\cdot a=(id_{\mathcal{A}}\otimes\tilde{\varphi})^*(f\cdot a),\quad a\cdot(id_{\mathcal{A}}\otimes\tilde{\varphi})^*( f)=\varphi(a)(id_{\mathcal{A}}\otimes\tilde{\varphi})^*(f).
\end{equation}
Since $\varphi$ is $wk^*$-continuous and also $\sigma wc(\mathcal{A}_*)=\mathcal{A}_*$, by (\ref{e2.2}), $(id_{\mathcal{A}}\otimes\tilde{\varphi})^*(\mathcal{A}_*)\subseteq\sigma wc({\mathcal{A}}\hat{\otimes}\frac{\mathcal{A}}{L})^*$. So we obtain a $wk^*$-continuous left $\mathcal{A}$-module morphism
\begin{equation*}
\Psi:=((id_{\mathcal{A}}\otimes \tilde{\varphi})^*\vert_{\mathcal{A}_*})^*:(\sigma wc(\mathcal{A}\hat{\otimes}\frac{\mathcal{A}}{L})^*)^*\rightarrow\mathcal{A}.
\end{equation*}
Define $\gamma_{\alpha}=\Psi\circ\eta_{\alpha}$. So $\gamma_{\alpha}:\frac{\mathcal{A}}{L}\rightarrow\mathcal{A}$ is a left $\mathcal{A}$-module morphism for every $\alpha$. Note that $\gamma_{\alpha}$ is a net of non-zero maps. To see this first we show that
\begin{equation}\label{2.4}
 \varphi\circ\Psi\circ\Theta=\varphi\circ\pi_{\sigma wc}.
 \end{equation}
  For every $a,b\in{\mathcal{A}}$ we have
  \begin{equation*}
 \begin{split}
 \varphi\circ(id_{\mathcal{A}}\otimes\tilde{\varphi})\circ (id_{\mathcal{A}}\otimes q)(a\otimes b)&=\varphi\circ(id_{\mathcal{A}}\otimes\tilde{\varphi})(a\otimes b+L)=\varphi(a\varphi(b))\\&=\varphi(a)\varphi(b)=\varphi(ab)=\varphi\circ\pi(a\otimes b).
 \end{split}
  \end{equation*}
  So 
  \begin{equation}\label{e2.6}
 \varphi\circ (id_{\mathcal{A}}\otimes\tilde{\varphi})\circ (id_{\mathcal{A}}\otimes q)=\varphi\circ\pi.
  \end{equation}
   One can see that 
  \begin{equation}\label{e2.7}
  \Psi(\hat{v}\vert_{\sigma wc(\mathcal{A}\hat{\otimes}\frac{\mathcal{A}}{L})^*})=id_{\mathcal{A}}\otimes\tilde{\varphi}(v)\qquad(v\in{\mathcal{A}\hat{\otimes}\frac{\mathcal{A}}{L}}),
  \end{equation}
 and also 
 \begin{equation}\label{e2.8}
\pi_{\sigma wc}(\bar{u})=\pi(u)\qquad(u\in{\mathcal{A}\hat{\otimes}\mathcal{A}}).
 \end{equation}
Now for every $m\in{(\sigma wc(\mathcal{A}\hat{\otimes}\mathcal{A})^*)^*}$, there exists a net $(u_{\alpha})$ in $\mathcal{A}\hat{\otimes}\mathcal{A}$ such that $m=wk^*\hbox{-}\lim\limits_{\alpha}\bar{u}_{\alpha}$. Since the maps $\varphi$, $\Theta$, $\Psi$ and $\pi_{\sigma wc}$ are $wk^*$-continuous, (\ref{e222}), (\ref{e2.6}), (\ref{e2.7}) and (\ref{e2.8}) imply that
\begin{equation*}
\begin{split}
\varphi\circ\Psi\circ\Theta(m)&=\varphi\circ\Psi\circ\Theta(wk^*\hbox{-}\lim\limits_{\alpha}\bar{u}_{\alpha})=wk^*\hbox{-}\lim\limits_{\alpha}\varphi\circ\Psi\circ\Theta(\bar{u}_{\alpha})\\&=wk^*\hbox{-}\lim\limits_{\alpha}\varphi\circ (id_{\mathcal{A}}\otimes\tilde{\varphi})\circ (id_{\mathcal{A}}\otimes q)(u_{\alpha})=wk^*\hbox{-}\lim\limits_{\alpha}\varphi\circ\pi(u_{\alpha})\\&=wk^*\hbox{-}\lim\limits_{\alpha}\varphi\circ\pi_{\sigma wc}(\bar{u}_{\alpha})=\varphi\circ\pi_{\sigma wc}(m).
\end{split}
\end{equation*}
So by (\ref{2.4}) for every $x\in{\mathcal{A}}$
\begin{equation*}
\varphi\circ\gamma_{\alpha}(x+L)=\varphi\circ\Psi\circ\eta_{\alpha}(x+L)=\varphi\circ\Psi\circ\Theta\circ\rho_{\alpha}(x)=\varphi\circ\pi_{\sigma wc}\circ\rho_{\alpha}(x)\rightarrow\varphi(x)\neq0.
\end{equation*}
Choose $x_0\in{\mathcal{A}}$ such that $\varphi(x_0)=1$. Let $m_{\alpha}=\gamma_{\alpha}(x_{0}+L)$. So $(m_{\alpha})$ is a net in $\mathcal{A}$. By similar argument as in \cite[Theorem 3.9]{sah:2016},
\begin{equation*}
am_{\alpha}=\varphi(a)m_{\alpha}\qquad\hbox{and}\qquad\varphi(m_{\alpha})=1.
\end{equation*}
 Hence $\mathcal{A}$ is left $\varphi$-contractible.
\end{proof}
\begin{example}
    Consider the Banach algebra $\ell^1$ of all sequences $a=(a_n)$ of complex numbers with
    \begin{equation*}
    \Vert{a}\Vert=\sum\limits_{n=1}^\infty\vert a_n\vert<\infty,
    \end{equation*}
    and the following product
    \begin{equation*}
    (a\ast b)(n)=\left\{
    \begin{array}{ll}
    a(1)b(1)& \hbox{if}\quad n=1\\
    a(1)b(n)+b(1)a(n)+a(n)b(n)&\hbox{if}\quad n>1
    \end{array}
    \right.
    \end{equation*}
    for every $a,b\in{\ell^1}$. By \cite[Example 4.1]{Sha:17}, $\ell^1$ is not Johnson pseudo-Connes amenable. Since $\ell^1$ is unital, Theorem \ref{t2.2} (i) implies that $\ell^1$ is not approximately Connes-biprojective.
\end{example}

\section{Application to triangular Banach algebras}
Let $\mathcal{A}$ and $\mathcal{B}$ be  Banach algebras and let $X$ be a Banach
$\mathcal{A},\mathcal{B}$-module. That is, $X$ is a Banach left
$\mathcal{A}$-module and a Banach right $\mathcal{B}$-module that satisfy $(a\cdot x)\cdot b=a\cdot(x\cdot b)$ and $||a\cdot x\cdot b||\leq ||a||||x||||b||$ for every $a\in \mathcal{A}$, $b\in \mathcal{B}$ and $x\in X$.
Consider

$$Tri(\mathcal{A},\mathcal{B},X)=\left(\begin{array}{cc} \mathcal{A}&X\\
0&\mathcal{B}\\
\end{array}\right),$$
with the usual matrix operations and  $$||\left(\begin{array}{cc} a&x\\
0&b\\
\end{array}\right)||=||a||+||x||+||b||\quad(a\in \mathcal{A},x\in X, b\in \mathcal{B}),$$  $Tri(\mathcal{A},\mathcal{B},X)$ becomes a Banach algebra which is called a triangular Banach algebra.\\
Note that if $\mathcal{A}$ is a dual Banach algebra, then $Tri(\mathcal{A},\mathcal{A},\mathcal{A})$ is a dual Banach algebra with respect to the predual $\mathcal{A}_*{\oplus}_{\infty}\mathcal{A}_*{\oplus}_\infty\mathcal{A}_*$.
\begin{thm}\label{T2.7}
    Let $\mathcal{A}$ be a dual Banach algebra with a left approximate identity and let $\varphi\in{\Delta_{wk^*}(\mathcal{A})}$. Then $Tri(\mathcal{A},\mathcal{A},\mathcal{A})$ is not approximately Connes-biprojective.
\end{thm}
\begin{proof}
     Suppose conversely that $Tri(\mathcal{A},\mathcal{A},\mathcal{A})$ is approximately Connes-biprojective. Let $(e_{\alpha})$ be a left approximate identity for $\mathcal{A}$. Then $(\left(\begin{array}{cc} e_{\alpha}&0\\
     0&e_{\alpha}\\
     \end{array}
     \right))$ is a left approximate identity for $Tri(\mathcal{A},\mathcal{A},\mathcal{A})$. Define a character $\psi_{\varphi}$ on $Tri(\mathcal{A},\mathcal{A},\mathcal{A})$ by $\psi_{\varphi}(\left(\begin{array}{cc} a&b\\
     0&c\\
     \end{array}
     \right))=\varphi(c)$, for every $a,b$ and $c$ in $\mathcal{A}$. The $wk^*$-continuity of $\varphi$ implies that $\psi_{\varphi}$ is $wk^*$-continuous. Since $Tri(\mathcal{A},\mathcal{A},\mathcal{A})$ has a left approximate identity, it is easy to see that $\ker\psi_{\varphi}=\overline{Tri(\mathcal{A},\mathcal{A},\mathcal{A})\ker\psi_{\varphi}}$. Using Theorem \ref{t2.5}, $Tri(\mathcal{A},\mathcal{A},\mathcal{A})$ is left $\psi_{\varphi}$-contractible.
     
      Consider the closed ideal $I=\left(\begin{array}{cc} 0&\mathcal{A}\\
     0&\mathcal{A}\\
     \end{array}
     \right)$ in $Tri(\mathcal{A},\mathcal{A},\mathcal{A})$. We have $\psi_{\varphi}\vert_I\neq0$. So $I$ is left $\psi_{\varphi}$-contractible \cite[Proposition 3.8]{Nasr:11}. Thus there exists $\left(\begin{array}{cc} 0&i\\
     0&j\\
     \end{array}
     \right)\in{I}$ such that
     \begin{equation}\label{eq3.1}
  \left(\begin{array}{cc} 0&a\\
  0&b\\
  \end{array}
  \right)   \left(\begin{array}{cc} 0&i\\
     0&j\\
     \end{array}
     \right)=\psi_{\varphi}(\left(\begin{array}{cc} 0&a\\
     0&b\\
     \end{array}
     \right))\left(\begin{array}{cc} 0&i\\
     0&j\\
     \end{array}
     \right)=\varphi(b)\left(\begin{array}{cc} 0&i\\
     0&j\\
     \end{array}
     \right)\quad(a,b\in{\mathcal{A}}),
     \end{equation}
     and
     \begin{equation}\label{eq3.2}
     \psi_{\varphi}(\left(\begin{array}{cc} 0&i\\
     0&j\\
     \end{array}
     \right))=\varphi(j)=1.
     \end{equation}
     Choose $b\in{\mathcal{A}}$ such that $\varphi(b)=0$. (\ref{eq3.1}) implies that for every $\alpha$
     \begin{equation*}
     \left(\begin{array}{cc} 0&e_{\alpha}j\\
     0&bj\\
     \end{array}
     \right)=\left(\begin{array}{cc} 0&e_{\alpha}\\
     0&b\\
     \end{array}
     \right)\left(\begin{array}{cc} 0&i\\
     0&j\\
     \end{array}
     \right)=\varphi(b)\left(\begin{array}{cc} 0&i\\
     0&j\\
     \end{array}
     \right)=0.
     \end{equation*}
   Then $e_{\alpha}j=0$ for every $\alpha$. Since $j=\lim\limits_{\alpha}e_{\alpha}j$, we have $j=0$ which is a contradiction with (\ref{eq3.2}).
\end{proof}
Let $S$ be the set  of natural numbers $\mathbb{N}$ with the binary operation $(m,n)\longmapsto\max\{m,n\}$, where $m$ and $n$ are in $\mathbb{N}$. Then $S$ is a weakly cancellative semigroup, that is, for every $s,t\in{S}$ the set $\{x\in{S}:sx=t\}$ is finite \cite[Example 3.36]{Dales:2010}. So $\ell^1(S)$ is a dual Banach algebra with predual $c_{0}(S)$ \cite[Theorem 4.6]{Dales:2010}.
\begin{cor}
Let $S=(\mathbb{N},\max)$. Then $Tri(\ell^{1}(S),\ell^{1}(S),\ell^{1}(S))$ is not approximately Connes-biprojective.
    \end{cor}
\begin{proof}
    Since $\ell^1(S)$ has a unit $\delta_{1}$, $Tri(\ell^{1}(S),\ell^{1}(S),\ell^{1}(S))$ has a unit. Fixed $n_{0}\in{\mathbb{N}}$. Consider the character $\varphi_{n_{0}}:\ell^1(S)\rightarrow\mathbb{C}$ defined by $\varphi_{n_{0}}(\sum\limits_{i=1}^\infty a_{i}\delta_{i})=\sum\limits_{i=1}^{n_0}a_{i}$. We claim that $\varphi_{n_{0}}$ is $wk^*$-continuous. To see this suppose that $(x_\alpha)=(\sum\limits_{i=1}^\infty a^\alpha_{i}\delta_{i})$ is a net in $\ell^1(S)$ and $x_0=\sum\limits_{i=1}^\infty b_{i}\delta_{i}$ is an element in $\ell^1(S)$ such that $x_\alpha\overset{wk^*}{\longrightarrow}x_0$. Let $g=\sum\limits_{i=1}^{n_0}\delta_{i}$. It is clear that $g\in{c_0(S)}$. So $x_\alpha(g)\longrightarrow x_0(g)$. It implies that $\sum\limits_{i=1}^{n_0}a^\alpha_{i}\longrightarrow\sum\limits_{i=1}^{n_0}b_{i}$. Thus $\varphi_{n_{0}}(x_\alpha)\longrightarrow\varphi_{n_{0}}(x_0)$. By Theorem \ref{T2.7}, $Tri(\ell^{1}(S),\ell^{1}(S),\ell^{1}(S))$ is not approximately Connes-biprojective.
\end{proof}
Let $\mathcal{A}$ be a Banach algebra and let $E$ be a Banach $\mathcal{A}$-bimodule. An element $x\in{E}$ is called weakly almost periodic if the module maps $\mathcal{A}\longrightarrow{E}$; $a\mapsto{a\cdot{x}}$ and $a\mapsto{x\cdot{a}}$ are weakly compact. The set of all weakly almost periodic elements of $E$ is denoted by $WAP(E)$ which is a norm closed sub-bimodule of $E$ \cite[Definition 4.1]{Runde:2004}.
For a Banach algebra $\mathcal{A}$, Runde observed that $F(\mathcal{A})=WAP(\mathcal{A}^{\ast})^{\ast}$ is a dual Banach algebra with the first Arens product inherited from $\mathcal{A}^{\ast\ast}$. He also showed that $F(\mathcal{A})$ is a canonical dual Banach algebra associated to $\mathcal{A}$ \cite[Theorem 4.10]{Runde:2004}.

We recall that if $E$ is a Banach $\mathcal{A}$-bimodule, then $E^*$ is also a Banach $\mathcal{A}$-bimodule via the following actions
\begin{equation*}
(a\cdot f)(x)=f(x\cdot a),\qquad(f\cdot a)(x)=f(a\cdot x)\quad(a\in{\mathcal{A}},x\in{E},f\in{E^*}).
\end{equation*}
\begin{cor}
Let $S=(\mathbb{N},\min)$. Then $Tri(F(\ell^{1}(S)),F(\ell^{1}(S)),F(\ell^{1}(S)))$ is not approximately Connes-biprojective.
\end{cor}
\begin{proof}
    Since $\ell^1(S)$ has a bounded approximate identity $(\delta_n)_{n\geq1}$ \cite[Example 4.10]{Dales:2010}, $F(\ell^1(S))$ has a unit \cite[Lemma 2.9]{Shari:17}. It follows that $Tri(F(\ell^{1}(S)),F(\ell^{1}(S)),F(\ell^{1}(S)))$ has a unit. Fixed $n_{0}\in{\mathbb{N}}$. Consider the character $\varphi_{n_{0}}:\ell^1(S)\rightarrow\mathbb{C}$ defined by $\varphi_{n_{0}}(\sum\limits_{i=1}^\infty a_{i}\delta_{i})=\sum\limits_{i=n_0}^{\infty}a_{i}$. It is easy to see that $a\cdot\varphi_{n_0}=\varphi_{n_0}(a)\varphi_{n_0}$ for every $a\in{\ell^1(S)}$. Let $\{a_{n}\}$ be a bounded sequence in $\ell^1(S)$. Since $\varphi_{n_0}$ is a bounded linear functional on $\ell^1(S)$, $\{\varphi_{n_0}({a}_{n})\}$ is a bounded sequence in $\mathbb{C}$. So there exists a convergence subsequence $\{\varphi_{n_0}({a}_{{n}_{k}})\}$ in  $\mathbb{C}$. Thus $\{\varphi_{n_0}({a}_{{n}_{k}})\}\varphi_{n_0}$ converges in $\ell^\infty(S)$ with the norm $\Vert\cdot\Vert_\infty$. Then ${a_{{n}_{k}}}\cdot\varphi_{n_0}$ converges weakly in $\ell^\infty(S)$. Hence the map $\ell^1(S)\rightarrow\ell^\infty(S)$, $a\mapsto a\cdot\varphi_{n_0}$ is weakly compact. So $\varphi_{n_0}\in{WAP(\ell^\infty(S))}$ \cite[Lemma 5.9]{Choi:2014}. Define $\tilde{\varphi}:F(\ell^1(S))\longrightarrow\mathbb{C}$ by $\tilde{\varphi}(X)=X(\varphi_{n_0})$ for every $X\in{F(\ell^1(S))}$. We claim that $\tilde{\varphi}$ is a $wk^*$-continuous character on $F(\ell^1(S))$. For every $X,Y\in{F(\ell^1(S))}$ we have
    \begin{equation*}
    \tilde{\varphi}(X\square Y)=X\square Y(\varphi_{n_0})=\langle Y\cdot\varphi_{n_0},X\rangle,
    \end{equation*}
    and also
    \begin{equation*}
    \langle a,Y\cdot\varphi_{n_0}\rangle=\langle\varphi_{n_0}\cdot a,Y\rangle=\langle\varphi_{n_0}(a)\varphi_{n_0},Y\rangle=\tilde{\varphi}(Y)\varphi_{n_0}(a)\qquad(a\in{\ell^1(S)}).
    \end{equation*}
    So $\tilde{\varphi}(X\square Y)=\langle\tilde{\varphi}(Y)\varphi_{n_0},X\rangle=\tilde{\varphi}(X)\tilde{\varphi}(Y)$. Let $(X_\alpha)$ be a net in $F(\ell^1(S))$ and let $X_0$ be an element in $F(\ell^1(S))$ such that $X_\alpha\overset{wk^*}{\longrightarrow}X_0$. Since $\varphi_{n_0}\in{WAP(\ell^\infty(S))}$, we have $X_\alpha(\varphi_{n_0}){\longrightarrow}X_0(\varphi_{n_0})$. Using Theorem \ref{T2.7}, $Tri(F(\ell^{1}(S)),F(\ell^{1}(S)),F(\ell^{1}(S)))$ is not approximately Connes-biprojective.
\end{proof}

\section{Application to some Banach algebras related to a locally compact group}
Let $G$ be a locally compact group. It is well-known that the measure algebra  $M(G)$ is a dual Banach algebra \cite[Example 4.4.2]{Runde:2002}.
\begin{thm}
For a locally compact group $ G $, the followings are equivalent:
\item[(i)] $ G $ is amenable,
\item[(ii)] $ M(G) $ is approximately Connes-biprojective.
\end{thm}
\begin{proof}
 (i) $\Rightarrow$ (ii) : If $ G $ is amenable, then $ M(G) $ is Connes amenable \cite[Theorem 5.4]{Runde:2003}. Applying \cite[Theorem 2.2]{Shi:2016}, $ M(G) $ is Connes-biprojective. So $ M(G) $ is approximately Connes-biprojective.\\
 (ii) $\Rightarrow$ (i) : If $ M(G) $ is approximately Connes-biprojective, then by Theorem \ref{t2.2} (i), $ M(G) $ is Johnson pseudo Connes-amenable. It follows that $ G $ is amenable \cite[Proposition 3.1]{Sha:17}.
\end{proof}
 Zhang showed that the Banach algebra $\ell^{2}(S)$ with the pointwise multiplication is approximately biprojective but it is not biprojective \cite[\textsection2]{zhan:99}. We extend this example to the approximately Connes-biprojective case.
\begin{thm}\label{ex3.1}
Let $S$ be an infinite set. Then $\ell^{2}(S)$ is approximately Connes-biprojective but it is not Connes-biprojective.
\end{thm}
\begin{proof}
Since $\mathcal{A}=\ell^2(S)$ is a Hilbert space, one can see that it is a dual Banach algebra. We know that  $\mathcal{A}$ is approximately biprojective \cite[\textsection2]{zhan:99}. Thus $\mathcal{A}$ is approximately Connes-biprojective. We claim that $\mathcal{A}$ is not Connes-biprojective.  Suppose conversely that there exists a  continuous $\mathcal{A}$-bimodule morphism  $\rho :\mathcal{A}\rightarrow(\sigma wc(\mathcal{A}\hat{\otimes}\mathcal{A})^*)^*$  such that $\pi_{\sigma wc}\circ\rho=id_{\mathcal{A}}$.  Composing the canonical inclusion
    map $\mathcal{A}\hat{\otimes}\mathcal{A}\hookrightarrow(\mathcal{A}\hat{\otimes}\mathcal{A})^{\ast\ast}$ with the quotient map $q:(\mathcal{A}\hat{\otimes}\mathcal{A})^{\ast\ast}\rightarrow(\sigma\omega{c}({\mathcal{A}}\hat{\otimes}{\mathcal{A}})^{\ast})^{\ast}$ as in the Remark \ref{R2.1}, we obtain a continuous $\mathcal{A}$-bimodule map $\tau:\mathcal{A}\hat{\otimes}\mathcal{A}\longrightarrow(\sigma\omega{c}({\mathcal{A}}\hat{\otimes}{\mathcal{A}})^{\ast})^{\ast}$ which has a $wk^{*}$-dense range. So there exists a bounded net $(u_{\alpha})$ in $(\mathcal{A}\hat{\otimes}\mathcal{A})$ such that
    \begin{equation*}
\rho(e_i)=wk^\ast\mbox{-}\lim_{\alpha}\tau(u_{\alpha})=wk^\ast\mbox{-}\lim_{\alpha}(\hat{u}_{\alpha})\vert_{\sigma\omega{c}({\mathcal{A}}\hat{\otimes}{\mathcal{A}})^{\ast}},
    \end{equation*}
    where for every $i\in{S}$,  $e_i$ is an element of $\mathcal{A}$ which is equal to $1$ at $i$ and $0$ elsewhere.
Similar argument as in \cite[Proposition 2.13]{Shari:17} leads us to a contradiction. Since $\rho(e_i)=e_i\cdot\rho(e_i)\cdot e_i$, one can see that
\begin{equation*}
\begin{split}
\rho(e_i)&=wk^{\ast}\hbox{-}\lim\limits_{\alpha}e_i\cdot\tau(u_{\alpha})\cdot e_i=wk^{\ast}\hbox{-}\lim\limits_{\alpha}\tau(e_i\cdot u_{\alpha}\cdot e_i)\\&=wk^{\ast}\hbox{-}\lim\limits_{\alpha}\tau(\lambda_{\alpha}{e}_i\otimes{e}_i)=wk^{\ast}\hbox{-}\lim\limits_{\alpha}\lambda_{\alpha}\tau({e}_i\otimes{e}_i),
\end{split}
\end{equation*}
  for some $(\lambda_{\alpha})\subseteq{\mathbb{C}}$. Since $\pi_{\sigma wc}$ is $wk^*$-continuous, 
\begin{equation*}
\begin{split}
{e}_i=\pi_{\sigma wc}\circ\rho(e_i)&=wk^{\ast}\hbox{-}\lim\limits_{\alpha}\lambda_{\alpha}\pi_{\sigma wc}((\widehat{{e}_i\otimes{e}_i})\vert_{\sigma{wc}({\mathcal{A}}\hat{\otimes}{\mathcal{A}})^{\ast}})\\&=wk^{\ast}\hbox{-}\lim\limits_{\alpha}\lambda_{\alpha}\pi({e}_i\otimes{e}_i)=wk^{\ast}\hbox{-}\lim\limits_{\alpha}\lambda_{\alpha}{e}_i.
\end{split}
\end{equation*}
So  $\lambda_{\alpha}\overset{\vert\cdot\vert}{\longrightarrow}1$ in $\mathbb{C}$. So $\rho(e_i)=\tau({e}_i\otimes{e}_i)$. Consider the identity operator $I:\mathcal{A}\rightarrow\mathcal{A}$, which can be viewed as an element of $(\mathcal{A}\hat{\otimes}\mathcal{A})^*$ \cite[\textsection3]{Daws:2006}. Define the map $\Phi:\mathcal{A}\hat{\otimes}\mathcal{A}\longrightarrow\mathcal{A}$ by $\Phi(a\otimes{b})=aI(b)$. We claim that $\Phi$ is weakly compact. We know that the unit ball of $\mathcal{A}\hat{\otimes}\mathcal{A}$ is the closure of the convex hull of $\{a\otimes{b} \ \colon \ \Vert{a}\Vert=\Vert{b}\Vert\leq1\}$. Since in a reflexive Banach space every bounded set
is relatively weakly compact, the set $\{ab \ \colon \ \Vert{a}\Vert=\Vert{b}\Vert\leq1\}$ is relatively weakly compact. So $\Phi$ is weakly compact. Applying \cite[Proposition 6.6]{Daws:2007}, we have $I\in{\sigma wc(\mathcal{A}\hat{\otimes}\mathcal{A})^*}$. If $x=\sum\limits_{i\in{S}}\beta_{i}e_{i}$ is an element in $\mathcal{A}$, then $\rho(x)=\sum\limits_{i\in{S}}\beta_{i}(\widehat{{e}_i\otimes{e}_i}\vert_{\sigma{wc}({\mathcal{A}}\hat{\otimes}{\mathcal{A}})^{\ast}})$. So
\begin{equation}\label{e3.1}
\langle{I},\rho(x)\rangle=\sum\limits_{i\in{S}}\beta_{i}\langle I,{e}_i\otimes{e}_i\rangle=\sum\limits_{i\in{S}}\beta_{i}\langle I(e_i),e_i\rangle=\sum\limits_{i\in{S}}\beta_{i}.
\end{equation}
We have 
\begin{equation*}
\vert\langle{I},\rho(x)\rangle\vert\leq\Vert{I}\Vert\Vert\rho\Vert\Vert{x}\Vert<\infty.
\end{equation*}
So by (\ref{e3.1}), $\sum\limits_{i\in{S}}\beta_{i}$ converges for every $x=\sum\limits_{i\in{S}}\beta_{i}e_{i}$ in ${\mathcal{A}}$. Then $\ell^2(S)\subset\ell^1(S)$, which is a contradiction with \cite[Proposition 6.11]{Fol:99}.
\end{proof}
Let $G$ be a locally compact group. Rickert showed that $L^2(G)$
is a Banach algebra with convolution if and only if $G$ is compact \cite{Rick:68}.

The proof of the following Corollary is similar to the \cite[Theorem 2.16]{Shari:17} so we omit it:
\begin{cor}
Let $G$ be an infinite commutative compact group. Then $L^2(G)$ with convolution is approximately Connes-biprojective, but it is not Connes-biprojective.
\end{cor}


\end{document}